\documentclass[12pt]{article}
\usepackage{amssymb}
\usepackage{amsfonts}
\usepackage{amsmath}
\usepackage[usenames]{color}
\usepackage{mathrsfs}
\usepackage{amsfonts}
\usepackage{amssymb,amsmath}
\usepackage{CJK}
\usepackage{cite}
\usepackage{cases}
\usepackage{amsthm}

\def\cl{\centerline}

\def\vs{\vspace*}
\def\W{\mathcal{L}}

\def\Z{\mathbb{Z}}

\def\C{\mathbb{C}}

\numberwithin{equation}{section}
\newtheorem{theo}{Theorem}[section]
\newtheorem{defi}[theo]{Definition}
\newtheorem{coro}[theo]{Corollary}
\newtheorem{lemm}[theo]{Lemma}
\newtheorem{exam}[theo]{Example}
\newtheorem{prop}[theo]{Proposition}

\newtheorem{case}{Case}

\newtheorem{rema}[theo]{Remark}

\begin{document}
\begin{center}
{\bf\large Finite irreducible conformal modules over the extended Block type Lie conformal algebra $\mathfrak{B}(\alpha,\beta,p)$}
\footnote {

Corresponding author: Y. Hong (hongyanyong2008@yahoo.com).
}
\end{center}

\cl{Haibo Chen, Yanyong Hong and    Yucai Su
}

\vs{8pt}

{\small\footnotesize
\parskip .005 truein
\baselineskip 3pt \lineskip 3pt
\noindent{{\bf Abstract:} In this paper, we  introduce a class of infinite Lie conformal algebras $\mathfrak{B}(\alpha,\beta,p)$, which are the semi-direct sums of Block type Lie conformal algebra $\mathfrak{B}(p)$
and its non-trivial conformal modules of $\Z$-graded free intermediate series. The  annihilation algebras  are a class of infinite-dimensional Lie algebras, which include
a lot of interesting subalgebras: Virasoro algebra, Block type Lie algebra, twisted Heisenberg-Virasoro algebra and so on.
  We give a
complete classification of all finite non-trivial irreducible conformal modules of  $\mathfrak{B}(\alpha,\beta,p)$ for $\alpha,\beta\in\C, p\in\C^*$.
As an application,  the classifications of finite irreducible conformal modules over a series of finite Lie conformal algebras  $\mathfrak{b}(n)$ for $n\geq1$ are given.
  \vs{5pt}

\noindent{\bf Key words:} Lie conformal algebra,  finite  conformal module, irreducible
\parskip .001 truein\baselineskip 6pt \lineskip 6pt

\noindent{\it Mathematics Subject Classification (2010):}  17B10,  17B65, 17B68.}}
\parskip .001 truein\baselineskip 6pt \lineskip 6pt

\section{Introduction}
The concept of  Lie conformal algebra  was   introduced  by Kac in \cite{K1,K3},
which gives
 an axiomatic description of the operator product expansion
of chiral fields in conformal field theory (see \cite{BPZ}).
 The theory of  Lie conformal algebra  plays an important role in  quantum field theory and vertex algebras.
 Furthermore, Lie conformal algebra  has
close connections to Hamiltonian formalism in the theory of nonlinear evolution (see \cite{BDK}).
In particular, they provide us powerful
tools for   the realization of the program of the study of Lie
(super)algebras and associative algebras (and their representations), satisfying the sole locality property (see \cite{K2}).

 A Lie conformal algebra is called finite if it is finite generated as a $\C[\partial]$-module. Otherwise, it is called infinite. Virasoro Lie conformal algebra $\mathfrak{Vir}$  and current
Lie conformal algebra $\mathrm{Cur}\mathfrak{g}$ associated to a finite-dimensional simple Lie algebra $\mathfrak{g}$ are two classes of important finite Lie
conformal algebras. As is well known,   $\mathfrak{Vir}$ and all current
Lie conformal algebra $\mathrm{Cur}\mathfrak{g}$  exhaust all
finite simple Lie conformal algebras (see \cite{DK}).
In recent years, the structure theory and representation theory of  finite Lie conformal algebras were intensively studied
(see, e.g., \cite{BKV, CK, CK1, DK, WY, LHW, SY1, YL}).

But, the theory of infinite Lie conformal algebra is  relatively backward.
Some interesting examples of  infinite  Lie conformal algebras  were  constructed  by closely linked infinite-dimensional loop   Lie   algebras,
such as loop Virasoro Lie conformal algebra,  loop Heisneberg-Virasoro  Lie conformal algebra, loop Schr\"{o}dinger-Virasoro Lie conformal algebra (see, e.g.,  \cite{WCY,CHSX,FSW}).
 One of the most important examples of  infinite simple conformal
algebras is the general Lie conformal algebra $gc_N$,  which  plays the same important role in
the theory of Lie conformal algebras as the general Lie algebra $gl_N$ does in the theory of Lie algebras. Thus, the general Lie conformal algebra $gc_N$ and its subalgebras have been investigated by many authors (see, e.g., \cite{BM,BKL,SY2,S}). In addition, there are also some infinite simple Lie conformal algebras constructed from Gel'fand-Dorfman bialgebras (see \cite{HW}).

In the present paper, we  define  a new  class of infinite  Lie conformal algebras $\mathfrak{B}(\alpha,\beta,p)$,
 which are associated with  Block type Lie conformal algebras $\mathfrak{B}(p)$ studied in \cite{SXY}.
Then  we determine the classification of all finite non-trivial irreducible conformal modules of  them.
{\it Block type Lie conformal algebras $\mathfrak{B}(p)$}   with $p\neq0$ has a $\C[\partial]$-basis
 $\{L_i\mid i\in\Z_+\}$ and $\lambda$-brackets as follows
\begin{eqnarray}\label{1.1}
 [L_i\, {}_\lambda \, L_j]=\big((i+p)\partial+(i+j+2p)\lambda \big) L_{i+j}.
\end{eqnarray}
Note that the subalgebra
$\mathfrak{Vir}=\C [\partial](\frac{1}{p}L_0)$
of $\mathfrak{B}(p)$ is the  so-called  Virasoro Lie conformal algebra.  All
finite irreducible conformal modules over $\mathfrak{Vir}$  were  explicitly  classified and constructed in \cite{CK}.
 The special case  $\mathfrak{B}(1)$ has a close relation with the general Lie conformal algebra $gc_1$.
It is worth to point out that the super analogue  of $\mathfrak{B}(p)$ was also constructed in \cite{X}   by analyzing certain module structures of $\mathfrak{B}(p)$.

 The rest of this paper
is organized as follows.
 In Section $2$, we introduce some
basic definitions, notations, and related known results about Lie conformal algebras.
In Section 3, we first introduce the definition of  $\mathfrak{B}(\alpha,\beta,p)$ by analyzing certain module structures of $\mathfrak{B}(p)$, and investigate  its  subalgebras, quotient algebras and  extended
annihilation algebras.
In Section 4,  we determine the irreducibility of all free
non-trivial rank one modules over $\mathfrak{B}(\alpha,\beta,p)$. Then we give a complete classification
of all finite non-trivial irreducible conformal modules of $\mathfrak{B}(\alpha,\beta,p)$ by showing that they must be free of rank one.
In Section $5$, we construct a class of new  Lie conformal superalgebras about $\mathfrak{B}(\alpha,\beta,p)$, which are generalizations
 of Lie conformal superalgebras  of Block type.
At last, as an application of our main result, we also obtain the classification of all finite non-trivial irreducible conformal modules  over $\mathfrak{b}(n)$ which are some quotient algebras of $\mathfrak{B}(\alpha,\beta,p)$.

Throughout this paper, all vector spaces, linear maps and tensor products are considered to be over
the field  of complex numbers. We denote by $\C,$ $\C^*$, $\Z$ and $\Z_+$ the sets of complex numbers, nonzero complex numbers,  integers and nonnegative integers, respectively.

\section{Preliminaries}
In this section, we recall some basic definitions and results related to Lie conformal algebras in \cite{DK,K1,K3} for later use.

\begin{defi}{\rm (\!\!\cite{K3})}\label{D1}
A {\it Lie conformal algebra} is a $\C[\partial]$-module $R$ endowed with a  $\lambda$-bracket $[a{}\, _\lambda \, b]$
which defines a
linear map $R\otimes R\rightarrow R[\lambda]$, where $\lambda$ is an indeterminate and $R[\lambda]=\C[\lambda]\otimes R$, subject to the following axioms:
\begin{equation*}
\aligned
&[\partial a\,{}_\lambda \,b]=-\lambda[a\,{}_\lambda\, b],\ \ \ \
[a\,{}_\lambda \,\partial b]=(\partial+\lambda)[a\,{}_\lambda\, b]\quad    {\rm(conformal\ sesquilinearity)},\\
&[a\, {}_\lambda\, b]=-[b\,{}_{-\lambda-\partial}\,a]\quad {\rm (skew\text{-}symmetry)},\\
&[a\,{}_\lambda\,[b\,{}_\mu\, c]]=[[a\,{}_\lambda\, b]\,{}_{\lambda+\mu}\, c]+[b\,{}_\mu\,[a\,{}_\lambda \,c]]\quad{\rm(Jacobi\ identity)}
\endaligned
\end{equation*}
for all $a,b,c\in R.$
\end{defi}

A Lie conformal algebra is called {\it finite} if it is finite generated as a $\C[\partial]$-module. Otherwise, it is called {\it infinite}.

\begin{defi}\label{defi-module}{\rm (\!\! \cite{CK})} A  conformal module $M$ over a Lie conformal algebra $R$ is a  $\C[\partial]$-module endowed with a $\lambda$-action $R\otimes M\rightarrow M[\lambda]$ such that
\begin{equation*}
(\partial a)\,{}_\lambda\, v=-\lambda a\,{}_\lambda\, v,\ a{}\,{}_\lambda\, (\partial v)=(\partial+\lambda)a\,{}_\lambda\, v,\
a\,{}_\lambda\, (b{}\,_\mu\, v)-b\,{}_\mu\,(a\,{}_\lambda\, v)=[a\,{}_\lambda\, b]\,{}_{\lambda+\mu}\, v\end{equation*}
for all $a,b\in R$,  $v\in M$.
\end{defi}

Let $R$ be a Lie conformal algebra.  A conformal $R$-module $M$ is called {\it finite} if it is finitely generated over $\C[\partial]$.
The {\it rank} of a   conformal module   $M$  is its rank as
a $\C[\partial]$-module.  If   $R$-module $M$  has no non-trivial submodules,  conformal module $M$ is
called {\it irreducible}.  If $R\,{}_\lambda v = 0$,   the element $v\in M$ is
called {\it invariant}.
\begin{defi}\label{d2.3}
 An annihilation algebra $\mathcal{A}(R)$ of a Lie conformal algebra $R$ is a Lie algebra with $\C$-basis $\{a(n)\mid a\in R,n\in\Z_+\}$ and relations $\mathrm{(}$for any $a$, $b\in R$ and $k\in \C$$\mathrm{)}$
\begin{eqnarray}
 \label{2asd2.1} &&(ka)_{(n)}=ka_{(n)},~~~~(a+b)_{(n)}=a_{(n)}+b_{(n)},\\
&&\label{22.1}[a_{(m)},b_{(n)}]=\sum_{k\in\Z_+}{m\choose k}(a_{(k)}b)_{(m+n-k)},\ (\partial a)_{(n)}=-na_{n-1},
\end{eqnarray}
where $a_{(k)}b$  is called the $k$-th product, given by $[a\, {}_\lambda \, b]=\sum_{k\in\Z_+}\frac{\lambda^{k}}{k!}(a_{(k)}b) $. Furthermore, an extended annihilation algebra $\mathcal{A}(R)^e$ of $R$
 is defined by $\mathcal{A}(R)^e=\C \partial\ltimes \mathcal{A}(R)$  with   $[\partial,a_{(n)}]=-na_{n-1}$.
\end{defi}
Similar to the definition of the $k$-th product $a_{(k)}b$ of two elements $a, b \in R$, we can
also define {\it $k$-th actions} of $R$ on $M$ for each $j\in \Z_+$, i.e. $a_{(k)}v$ for any $a\in R, v\in M$
\begin{eqnarray}\label{22.111}
a\, {}_\lambda \, v=\sum_{k\in\Z_+}\frac{\lambda^{(k)}}{k!}(a_{(k)}v).
\end{eqnarray}
A close connection between the module of a Lie
conformal algebra and that of its extended annihilation algebra  was studied in \cite{CK} by Cheng and Kac.
\begin{prop}\label{pro2.4}
A conformal module $M$ over a Lie conformal algebra $R$ is the same as a module over the
 Lie algebra $\mathcal{A}(R)^e$  satisfying $a_{(n)}v=0$ for $a\in R,v\in M,n\gg 0$.
\end{prop}
 The following result can be found in  \cite{CK,K3}, which   plays  an important role in our classification.
\begin{lemm}\label{lemm2.5}
 Let $\W$ be a Lie superalgebra with a descending sequence of subspaces  $\W\supset\W_0\supset\W_1\supset\cdots$ and an element $\partial$ satisfying $[\partial, \W_n] =\W_{n-1}$ for $n\geq1$. Let $V$ be an $\W$-module and let
$$V_n=\{v\in V | \W_n v =0\},\  n\in\Z_+.$$
Suppose that $V_n\neq0$ for $n\gg0$  and let $N$ denote the minimal such $n$. Suppose that $N\geq1$. Then $V=\C[\partial]\otimes_{\C} V_N$. In particular, $V_N$ is finite-dimensional
 if $V$ is a finitely generated $\C[\partial]$-module.
\end{lemm}

\section{ Lie conformal algebra $\mathfrak{B}(\alpha,\beta,p)$}

In this section, we  define a class of extended Block type   Lie conformal algebras $\mathfrak{B}(\alpha,\beta,p)$
 by using Block type Lie conformal algebras $\mathfrak{B}(p)$ and  their  intermediate series modules.
Now we recall the definition of intermediate series modules of  $\mathfrak{B}(p)$ (see \cite{X}).

For  $\alpha,\beta\in\C,p \in\C^*$,   the   $\C[\partial]$-module $V(\alpha,\beta,p)=\bigoplus_{i\in\Z}\C[\partial]v_i$  is a $\Z$-graded free intermediate
series module over $\mathfrak{B}(p)$ with  $\lambda$-action as follows:
\begin{equation*}
 L_i\, {}_\lambda \, v_j=\Big((i+p)(\partial+\beta)+(i+j+\alpha)\lambda\Big)v_{i+j}.
\end{equation*}
Then we can  define  infinite Lie conformal algebra called
  {\it extended Block type Lie conformal algebra} $\mathfrak{B}(\alpha,\beta,p)$, which
 has a  $\C[\partial]$-basis $\{L_{i},W_{i}\mid i\in\Z_+ \}$
satisfying \eqref{1.1} and the following   $\lambda$-brackets
\begin{eqnarray}\label{3.1}
[L_i\, {}_\lambda \, W_j]=\Big((i+p)(\partial+\beta)+(i+j+\alpha)\lambda\Big) W_{i+j},\  [W_i\, {}_\lambda \, W_j]=0
\end{eqnarray}
for any $\alpha,\beta\in\C,p\in\C^*$.
Here,  $\mathfrak{B}(\alpha,\beta,p)$ is regarded as   a $\Z$-graded Lie conformal algebra. We can also introduce a class  of $\Z$-graded free intermediate series modules over $\mathfrak{B}(\alpha,\beta,p)$. Given $a,b,c\in\C$, let $V_{a,b,c}=\oplus_{i\in\Z}\C[\partial]v_i$ and
define
\begin{eqnarray*}
&&L_i\, {}_\lambda \, v_j=\Big((i+p)(\partial+b)+(i+j+a)\lambda\Big)v_{i+j},
\\&& W_i\, {}_\lambda \, v_j=\delta_{\alpha,p}\delta_{\beta,0}cv_{i+j}\nonumber\quad \mbox{for any} \ i,j\in\Z.
\end{eqnarray*}

Some interesting features on this class of Lie conformal algebras are presented as follows.
\subsection{Subalgebras}
Setting $\alpha^\prime=\frac{\alpha}{p},L=\frac{1}{p}L_0,W= W_0\in \mathfrak{B}(\alpha,\beta,p)$ in  \eqref{1.1} and \eqref{3.1},  we see that
$$[L\, {}_\lambda \, L]=(\partial+2\lambda)L,\  [L\, {}_\lambda \, W]=\big(\partial+\alpha^\prime\lambda+\beta\big)W,\ [W\, {}_\lambda \, W]=0$$
for $\alpha^\prime,\beta\in\C$.
Namely, the subalgebra
$$\C[\partial]L\oplus\C[\partial]W$$
 of $\mathfrak{B}(\alpha,\beta,p)$ is  the $\mathcal{W}(\alpha^\prime,\beta)$ Lie  conformal algebra  (see \cite{LHW}, $\beta=0$ also see \cite{WY}).
Here, we note that  $\mathcal{W}(1,0)$  and $\mathcal{W}(2,0)$ are respectively    Heisenberg-Virasoro  Lie conformal algebra and $W(2,2)$  Lie conformal algebra.
An interesting fact about  them are given.
Now we compute in  $\mathcal{W}(1,0)$.
For any $h\in\C^*,$ it is clear that
\begin{eqnarray*}
[(L+hW)\, {}_\lambda \, (L+hW)]=(\partial+2\lambda)(L+hW).
\end{eqnarray*}
Hence, $\C[\partial](L+hW)$ spans a  subalgebra of Heisenberg-Virasoro Lie conformal algebra and $W(2,2)$  Lie conformal algebra in common which is isomorphic to the classical Virasoro Lie conformal algebra.

Moreover, the Lie conformal algebra $\mathfrak{B}(\alpha,\beta,p)$ has a non-trivial abelian conformal ideal $\{W_i\mid i\in\Z_+\}$ as  a  $\C[\partial]$-module, which implies
that it is neither simple nor semi-simple.

\subsection{Quotient algebras}
Considering the quotient algebras of $\mathfrak{B}(\alpha,\beta,p)$, we will  get many finite Lie conformal algebras.
Note that $\mathfrak{B}(\alpha,\beta,p)$ is $\Z$-graded under the sense that $\mathfrak{B}(\alpha,\beta,p)=\bigoplus_{k\in\Z_+}\mathfrak{B}(\alpha,\beta,p)_k$,
 where $\mathfrak{B}(\alpha,\beta,p)_k=\C[\partial] L_k\oplus\C [\partial] W_k$. For $n\in\Z_+,$  define a subspace $\mathfrak{B}(\alpha,\beta,p)_{\langle n\rangle}$
of $\mathfrak{B}(\alpha,\beta,p)$ by
$$\mathfrak{B}(\alpha,\beta,p)_{\langle n\rangle}=\bigoplus_{i\geq n}\C[\partial]L_i\oplus\bigoplus_{i\geq n}\C[\partial]W_i.$$
It is clear that $\mathfrak{B}(\alpha,\beta,p)_{\langle n\rangle}$ is an ideal of  $\mathfrak{B}(\alpha,\beta,p)$.
For any $n\in\Z_+$, we define
\begin{eqnarray}\label{b3.2}
\mathfrak{B}(\alpha,\beta,p)_{[n]}=\mathfrak{B}(\alpha,\beta,p)/\mathfrak{B}(\alpha,\beta,p)_{\langle n+1\rangle}
\end{eqnarray}
 Note that $\mathfrak{B}(\alpha,\beta,p)_{[0]}\cong\mathcal{W}(\alpha^\prime,\beta)$. Taking $p=-n$,  we can define  the quotient algebras
   $\mathfrak{B}(\alpha,\beta,-n)_{[n]}$
 by the following relations
\begin{eqnarray}\label{bn3.2}
\mathfrak{b}(n)=\mathfrak{B}(\alpha,\beta,-n)_{[n]}=\mathfrak{B}(\alpha,\beta,-n)/\mathfrak{B}(\alpha,\beta,-n)_{\langle n+1\rangle}
\end{eqnarray}
with $n\geq1$. They can produce a series of new finite non-simple Lie conformal algebras.
Two examples for $n=1,2$ are presented as follows.
\begin{exam}
Setting $L=-\bar L_0,W=\bar W_0,M=\bar L_1,G=\bar W_1\in \mathfrak{b}(1)$, we have the following non-trivial relations
 \begin{eqnarray*}
&&[L\, {}_\lambda \, L]=(\partial+2\lambda)L,\  [L\, {}_\lambda \, W]=\big(\partial+\beta-\alpha\lambda\big)W
\\&&[L\, {}_\lambda \, M]=(\partial+\lambda)M, \ [L\, {}_\lambda \, G]=\big(\partial+\beta-(1+\alpha)\lambda\big)G
\\&&[M\, {}_\lambda \, W]= (1+\alpha)\lambda G.
\end{eqnarray*}
Other  $\lambda$-brackets   are given by skew-symmetry.  Note that  $\C[\partial]L\oplus \C[\partial]M$  and $\C[\partial]L\oplus \C[\partial]W$ are respectively    Heisenberg-Virasoro Lie conformal algebra and
   $\mathcal{W}(\alpha^\prime,\beta)$ Lie conformal algebra. Maybe  $\mathfrak{b}(1)$ should be called      Heisenberg-$\mathcal{W}(\alpha^\prime,\beta)$ Lie conformal algebra.
\end{exam}

\begin{exam}
Set  $L=-\frac{1}{2}\bar L_0,Y=\bar L_1,M=-\bar L_2,W=\bar W_0,G=\bar W_1,H=\bar W_2\in \mathfrak{b}(2)$.   The non-trivial $\lambda$-brackets are as follows
 \begin{eqnarray*}
&&[L\, {}_\lambda \, L]=(\partial+2\lambda)L,\  [L\, {}_\lambda \, W]=\big(\partial+\beta-\frac{1}{2}\alpha\lambda\big)W
\\&&[L\, {}_\lambda \, Y]=(\partial+\frac{3}{2}\lambda)Y, \ [L\, {}_\lambda \, G]=\big(\partial+\beta-\frac{1}{2} (1+\alpha)\lambda\big)G,
\\&&[L\, {}_\lambda \, M]=(\partial+\lambda)M,\ [L\, {}_\lambda \, H]=\big(\partial+\beta-\frac{1}{2}(2+\alpha)\lambda\big)H,
\\&&[Y\, {}_\lambda \, Y]=(\partial+2\lambda)M,\ [Y\, {}_\lambda \, W]= \big(-(\partial+\beta)+(1+\alpha)\lambda\big)G,
 \\&& [Y\, {}_\lambda \, G]= \big(-(\partial+\beta)+(2+\alpha)\lambda\big)H,\
[M\, {}_\lambda \, W]=-(2+\alpha) \lambda H.
\end{eqnarray*}
Other  $\lambda$-brackets   are given by skew-symmetry.  We note that  $\C[\partial]L\oplus \C[\partial]Y\oplus \C[\partial]M$  and $\C[\partial]L\oplus \C[\partial]W$ are respectively    Schr\"{o}dinger-Virasoro Lie conformal algebra and
   $\mathcal{W}(\alpha^\prime,\beta)$ Lie conformal algebra. Maybe  $\mathfrak{b}(2)$ should be called     Schr\"{o}dinger-$\mathcal{W}(\alpha^\prime,\beta)$ Lie conformal algebra.
\end{exam}

\subsection{Extended annihilation algebra}
Now   we   give the explicit Lie brackets of $\mathcal{A}(\mathfrak{B}(\alpha,\beta,p))$ and $\mathcal{A}(\mathfrak{B}(\alpha,\beta,p))^e$.
\begin{lemm}\label{4.1}
 \begin{itemize}\parskip-7pt
 \item[{\rm (1)}] The annihilation  algebra  of $\mathfrak{B}(\alpha,\beta,p)$ is
$$\mathcal{A}(\mathfrak{B}(\alpha,\beta,p))=\{L_{i,m},W_{j,n}\mid i,j\in\Z_+,m\in \Z_+\cup\{-1\}, n\in \Z_+\}$$
with the following      Lie brackets:
\begin{equation}\label{44.1}
\aligned
&[L_{i,m},L_{j,n}]=\big((m+1)(j+p)-(n+1)(i+p)\big)L_{i+j,m+n},
\\&[L_{i,m},W_{j,n}]=\big((m+1)(j-p+\alpha)-n(i+p)\big)W_{i+j,m+n}+\beta(i+p)W_{i+j,m+n+1},
\\&[W_{i,m},W_{j,n}]=0.
\endaligned
\end{equation}
 \item[{\rm (2)}] The extended annihilation algebra  is
 $$\mathcal{A}(\mathfrak{B}(\alpha,\beta,p))^e=\{L_{i,m},W_{j,n},\partial\mid i,j\in\Z_+,m\in \Z_+\cup\{-1\}, n\in \Z_+\}$$
satisfying
\eqref{44.1} and  $[\partial,L_{i,m}]=-(m+1)L_{i,m-1},[\partial,W_{j,n}]=-nW_{j,n-1}$.
\end{itemize}

\end{lemm}

\begin{proof}
It follows from the definition of the $k$-th product    in Definition  \ref{d2.3} and   $\mathfrak{B}(\alpha,\beta,p)$ that we have
\begin{eqnarray*}
L_i\,{}_{{}_{(k)}} L_j&=&
\begin{cases}
(i+p)\partial L_{i+j} &\ \mbox{if}\  k=0,\\[4pt]
(i+j+2p)L_{i+j}&\  \mbox{if} \ k=1,\\[4pt]
0&\  \mbox{if} \ k\geq2,
\end{cases}\\
L_i\,{}_{{}_{(k)}} W_j&=&
\begin{cases}
(i+p)(\partial+\beta)W_{i+j} &\ \mbox{if}\  k=0,\\[4pt]
(i+j+\alpha)  W_{i+j}&\  \mbox{if} \ k=1,\\[4pt]
0&\  \mbox{if} \ k\geq2,
\end{cases}\\
W_i\,{}_{{}_{(k)}} W_j&=&0 \quad \mathrm{for \ any}\quad k\in\Z_+.
\end{eqnarray*}
Then by \eqref{2asd2.1} and \eqref{22.1}, we check that:
 \begin{equation}\label{44.2}
\aligned
&[(L_i)_{(m)},(L_j)_{(n)}]=\big(m(j+p)-n(i+p)\big)(L_{i+j})_{(m+n-1)},
\\&[(L_i)_{(m)},(W_j)_{(n)}]=\big(m(j-p+\alpha)-n (i+p)\big)(W_{i+j})_{(m+n-1)}+\beta(i+p)(W_{i+j})_{(m+n)},
\\&[(W_i)_{(m)},(W_j)_{(n)}]=0,\ [\partial,(L_i)_{(m)}]=-m(L_i)_{(m-1)},\ [\partial,(W_j)_{(n)}]=-n(W_j)_{(n-1)}.
\endaligned
\end{equation}
 Setting $L_{i,m}=(L_i)_{(m+1)}, W_{j,n}=(W_j)_{(n)}$ in \eqref{44.2} for $i,j\in\Z_+,m\in\Z_+\cup\{-1\},n\in\Z_+$, the lemma holds.
\end{proof}
\begin{rema}
  The Lie algebra $\mathcal{A}\big(\mathfrak{B}(\alpha,\beta,p)\big)$ is interesting in the sense that it contains the following subalgebras:
  \begin{itemize}\parskip-7pt
 \item[{\rm (a)}] when $\alpha=p,\beta=0$, the well-known twisted Heisenberg-Virasoro algebra is isomorphic to the Lie algebra  spanned by   $\{L_{i,0},W_{j,0}\mid i,j\in\Z\}$;
 \item[{\rm (b)}] the Lie algebra generates by $\{L_{0,m},W_{0,n}\mid m,n\in\Z\}$ is isomorphic to the annihilation algebra of $\mathcal{W}(\alpha^\prime,\beta)$ in \cite{LHW}.
\end{itemize}

\end{rema}

Next, we construct a subquotient algebra of  $\mathcal{A}(\mathfrak{B}(\alpha,\beta,p))$ and study its representation theory. Clearly,
$$\mathcal{A}(\mathfrak{B}(\alpha,\beta,p))_+=\{L_{i,m},W_{j,n}\mid i,j,m,n\in \Z_+\}$$
is a subalgebra of $\mathcal{A}(\mathfrak{B}(\alpha,\beta,p))$. For any fixed $k,N\in\Z_+$,
$$\mathcal{I}(k,N)=\{L_{i,m},W_{j,n}\in\mathcal{A}(\mathfrak{B}(\alpha,\beta,p))_+\mid i,j>k,m,n>N\}$$
is an ideal of $\mathcal{A}(\mathfrak{B}(\alpha,\beta,p))_+$.
Denote  $$\mathcal{Q}(k,N)=\mathcal{A}(\mathfrak{B}(\alpha,\beta,p))_+/\mathcal{I}(k,N).$$
\begin{lemm}\label{lemm555}
Let $V$ be a non-trivial finite-dimensional irreducible module over $\mathcal{Q}(k,N)$. Then we have $\mathrm{dim}(V)=1.$
\end{lemm}
\begin{proof}
 It follows from Lie's
Theorem that we see that any irreducible finite-dimensional
module over the solvable Lie algebra  $\mathcal{Q}(k,N)$ is one-dimensional.
\end{proof}

\section{Classification of finite irreducible modules}
The aim of this section is to  give a complete  classification  of all finite non-trivial irreducible conformal modules over  $\mathfrak{B}(\alpha,\beta,p)$.
The main results will be presented after some preparations.
\subsection{Equivalence of modules}
 The following classification of finite non-trivial irreducible conformal modules over
 $\mathfrak{B}(p)$ appeared in \cite{SXY}, which will be used in the following.
\begin{lemm}\label{5.1}
Let $V$ be a finite non-trivial irreducible conformal module over  $\mathfrak{B}(p)$. Then $V$ is isomorphic to one of the following
 \begin{itemize}\parskip-7pt
 \item[{\rm (1)}] $V_{a,b}=\C[\partial]v$ with
 $$L_0\,{}_\lambda\, v=p(\partial+a\lambda+b)v$$
 for $a\in\C^*,b\in\C$, if  $p\neq-1$;
 \item[{\rm (2)}] $V_{a,b,c}=\C[\partial]v$ with
 $$L_0\,{}_\lambda\, v=-(\partial+a\lambda+b)v,\ L_1\,{}_\lambda\, v=cv$$ for  $a\in\C^*$  or  $c\in\C^*$, if $p=-1$.
\end{itemize}
\end{lemm}
Now we   give the equivalence between the finite conformal modules over $\mathfrak{B}(\alpha,\beta,p)$
 and those over its quotient algebra $\mathfrak{B}(\alpha,\beta,p)_{[n]}$ for  some $n\in\Z_+$.
 \begin{theo}\label{5.2}
Assume that $V$ is a finite non-trivial conformal module over $\mathfrak{B}(\alpha,\beta,p)$.
Then the $\lambda$-actions  of $L_i$ and $W_i$ on $V$ are trivial for $i\gg0$.
\end{theo}
\begin{proof}
Clearly, $V$ is also a finite  conformal module over  $\mathfrak{B}(p)$.
Using Lemma 3.1 of \cite{SXY}, we obtain
$L_i\,{}_\lambda\, v=0$  for all $i\gg0$ and any $v\in V$. Choose such $i$ such that $i>|\alpha|$.
Fix $i\gg0$.  Using
 $$L_i\,{}_\lambda\,(W_0{}\,_\mu\, v)-W_0\,{}_\mu\,(L_i\,{}_\lambda\, v)=\Big(\big((i+p)(\partial+\beta)+(i+\alpha)\lambda\big)W_{i}\Big) \,{}_{\lambda+\mu}\, v,$$
one has $W_{i} \,{}_{\lambda}\, v=0$
 for any $v\in V$. The theorem holds.
 \end{proof}

\begin{rema}
A finite conformal module over $\mathfrak{B}(\alpha,\beta,p)$  is isomorphic to  a   finite conformal module over
$\mathfrak{B}(\alpha,\beta,p)_{[n]}$  for some  large enough  $n\in\Z$, where $\mathfrak{B}(\alpha,\beta,p)_{[n]}$ is defined by \eqref{b3.2}.
\end{rema}

\subsection{Rank one modules}
Now  we  give a characterization of    non-trivial  free   conformal modules of rank one  over
$\mathfrak{B}(\alpha,\beta,p)$.
From Lemma \ref{5.1}, we can define two classes of conformal modules $V_{a,b},V_{a,b,d}$ and $V_{a,b,c},V_{a,b,c,d}$ as follows.
\begin{itemize}\parskip-7pt
 \item[{\rm (1)}] $V_{a,b}=\C[\partial]v$ with
 $$L_0\,{}_\lambda\, v=p(\partial+a\lambda+b)v,\ W_0\,{}_\lambda\, v=W_i\,{}_\lambda\,v=L_i\,{}_\lambda\,v=0,\ i\geq1$$
 for $a,b\in\C$, if  $(\alpha,\beta)\neq(p,0)$;
 \item[{\rm (2)}] $V_{a,b,d}=\C[\partial]v$ with
 $$L_0\,{}_\lambda\, v=p(\partial+a\lambda+b)v,\ W_0\,{}_\lambda\, v=dv,\ W_i\,{}_\lambda\,v=L_i\,{}_\lambda\,v=0,\ i\geq1$$ for  $a,b,d\in\C$, if $(\alpha,\beta)=(p,0)$.
\end{itemize}
In fact, $V_{a,b}$ and  $V_{a,b,d}$ are just conformal modules over $\mathcal{W}(\alpha^\prime,\beta)$ (see \cite{LHW}).
\begin{itemize}\parskip-7pt
 \item[{\rm (3)}] $V_{a,b,c}=\C[\partial]v$ with
 $$L_0\,{}_\lambda\, v=-(\partial+a\lambda+b)v,\ L_1\,{}_\lambda\,v=cv,\ W_i\,{}_\lambda\, v=L_j\,{}_\lambda\,v=0,\ i\geq1,j\geq2$$
 for $a,b,c\in\C$, if  $(\alpha,\beta)\neq(-1,0)$;
 \item[{\rm (4)}] $V_{a,b,c,d}=\C[\partial]v$ with
 $$L_0\,{}_\lambda\, v=-(\partial+a\lambda+b)v,\ L_1\,{}_\lambda\,v=cv,\   W_0\,{}_\lambda\, v=dv, \    W_i\,{}_\lambda\,v=L_j\,{}_\lambda\,v=0,\ i\geq1,j\geq2$$ for  $a,b,c,d\in\C$, if $(\alpha,\beta)=(-1,0)$.
\end{itemize}
For $\mathfrak{B}(\alpha,\beta,-1)$,  we see that $V_{a,b,c}$
and  $V_{a,b,c,d}$  are just $\mathfrak{B}(p)$-conformal  modules  if $d=0$
(see \cite{SXY}).

\begin{theo}\label{77.1}
Let $V$ be a non-trivial free   conformal module of rank one  over  $\mathfrak{B}(\alpha,\beta,p)$.
\begin{itemize}\parskip-7pt
 \item[{\rm (1)}]
  If $p\neq-1$, and
 \begin{eqnarray*}
\begin{cases}
{\rm (i)}\ (\alpha,\beta)\neq(p,0), \ \mathrm{then}\  V\cong V_{a,b}  &\ \mbox{with}\  a,b\in\C,\\[4pt]
{\rm (ii)}\ (\alpha,\beta)=(p,0),  \ \mathrm{then}\ V\cong V_{a,b,d} &\  \mbox{with} \ a,b,d\in\C;
\end{cases}\\
\end{eqnarray*}
 \item[{\rm (2)}]
  If $p=-1$, and
 \begin{eqnarray*}
\begin{cases}
{\rm (iii)}\ (\alpha,\beta)\neq(-1,0), \ \mathrm{then}\  V\cong V_{a,b,c}  &\ \mbox{with}\  a,b,c\in\C,\\[4pt]
{\rm (iv)}\ (\alpha,\beta)=(-1,0),  \ \mathrm{then}\ V\cong V_{a,b,c,d} &\  \mbox{with} \ a,b,c,d\in\C.
\end{cases}\\
\end{eqnarray*}
\end{itemize}
\end{theo}

\begin{proof}
Let $V$ be a $\C[\partial]$-module.  Regarding $V$ as a conformal module over $\mathfrak{B}(p)$ and according to the result in \cite{SXY}, it is clear that
\begin{eqnarray*}
\begin{cases}
 \ L_0\,{}_\lambda\, v=p(\partial+a\lambda+b),\ L_i\,{}_\lambda\, v=0,\  a,b\in\C,i\geq1  &\ \mbox{if}\   p\neq-1;\\[4pt]
\ L_0\,{}_\lambda\, v=p(\partial+a\lambda+b),\ L_1\,{}_\lambda\, v=cv, \ L_i\,{}_\lambda\, v=0,\ a,b,c\in\C,i\geq2 &\  \mbox{if} \ p=-1.
\end{cases}\\
\end{eqnarray*}
By Theorem \ref{5.2}, we see that   $L_i\,{}_\lambda\, v=W_i\,{}_\lambda\, v=0$   for $i\gg0$.
 Let $k\in\Z_+$ be the largest integer such that the action of $\mathfrak{B}(\alpha,\beta,p)_k$ on $V$ is non-trivial.
By the assumption of $k$, we can suppose that
$$L_k\,{}_\lambda\, v=g_k(\partial,\lambda)v, \ W_k\,{}_\lambda\, v=h_k(\partial,\lambda)v,$$
where $g_k(\partial,\lambda),h_k(\partial,\lambda)\in\C[\partial, \lambda]$ and at least one of them is nonzero.
We note that $g_0(\partial,\lambda)=p(\partial+a\lambda+b),g_i(\partial,\lambda)=0$  for $i\geq2$ and
\begin{eqnarray*}
\begin{cases}
 \  g_1(\partial,\lambda)=0   &\ \mbox{if}\   p\neq-1;\\[4pt]
\  g_1(\partial,\lambda)=c   &\  \mbox{if} \ p=-1.
\end{cases}\\
\end{eqnarray*} For $i,j\in\Z_+$, by   Definition   \ref{defi-module}, one has
 \begin{eqnarray*}
 &&
   L_i\,{}_\lambda\, (W_j{}\,_\mu\, v)-W_j\,{}_\mu\,(L_i\,{}_\lambda\, v)=[L_i\,{}_\lambda\, W_j]\,{}_{\lambda+\mu}\, v,
 \ W_i\,{}_\lambda\, (W_j{}\,_\mu\, v)-W_j\,{}_\mu\,(W_i\,{}_\lambda\, v)=0,
\end{eqnarray*}
which gives that
 \begin{eqnarray}
&& \label{6.5}   \nonumber h_{j}(\partial+\lambda,\mu)g_{i}(\partial,\lambda)-g_{i}(\partial+\mu,\lambda)h_{j}(\partial,\mu)
 \\&=&\big((i+p)(\beta-\mu)+(j+\alpha-p)\lambda\big) h_{i+j}(\partial,\lambda+\mu),
 \\&&
 \label{6.6} \nonumber h_{j}(\partial+\lambda,\mu)h_{i}(\partial,\lambda)-h_{i}(\partial+\mu,\lambda)h_{j}(\partial,\mu)
 \\&=&0.
\end{eqnarray}
Setting $i=j=k$ in   \eqref{6.6}, and  comparing the highest degree of $\lambda$, we get
 $h_{k}(\partial,\lambda)=h_{k}(\lambda)$ for any $k\in\Z_+$.

\begin{case}
 $k=0$.
 \end{case}
 Then setting  $i=j=0$ in    \eqref{6.5}, we check that
 \begin{eqnarray}
\label{66.88}p\mu  h_0(\mu)+ \big(p(\beta-\mu)+(\alpha-p)\lambda\big) h_{0}(\lambda+\mu)=0.
\end{eqnarray}
Choosing  $\mu=0$ in \eqref{66.88}, if $\alpha\neq p$ or $\beta\neq0$, we have $h_{0}(\lambda)=0$.
Considering $\alpha=p,\beta=0$ in \eqref{66.88}, we obtain that $h_{0}(\lambda+\mu)=h_{0}(\mu)$, which implies $h_{0}(\lambda+\mu)=h_{0}(\mu)=d\in\C$.
\begin{case}
 $k\geq1$.
 \end{case}
Note that $g_k(\partial,\lambda),h_0(\partial,\lambda)\in\C$ for $k\geq1$. Taking  $i=k,j=0$ in    \eqref{6.5}, one can get  that
 \begin{eqnarray}
\label{6.8} \big((k+p)(\beta-\mu)+(\alpha-p)\lambda\big) h_{k}(\lambda+\mu)=0.
\end{eqnarray}
If $p\neq-k$, we immediately obtain  $h_{k}(\lambda)=0$. Let $p=-k$. In  \eqref{6.8}, if   $\alpha\neq -k$, one has $h_{k}(\lambda)=0$.
If $\alpha=-k=p$, we consider $i=0,j=k$ in  \eqref{6.5}, it can be rewritten as
 \begin{eqnarray}
 \label{6.9} p(\beta-\lambda-\mu) h_{k}(\lambda+\mu)=-p\mu h_{k}(\mu),
\end{eqnarray}
which implies $h_{k}(\lambda)=0$. Now we always have $h_{k}(\partial,\lambda)=0$ for $k\geq1$.
Thus, we can conclude that the action of $\mathfrak{B}(\alpha,\beta,p)_k$ on $V$ is  trivial for $k\geq2$.

Then we can directly obtain this theorem by the discussion above and the classification of rank one modules of $\mathfrak{B}(\alpha,\beta,p)$. This completes the proof.
\end{proof}

The irreducibilities of conformal  modules $V$ over $\mathfrak{B}(\alpha,\beta,p)$ defined  in Theorem  \ref{77.1}
are easy to determine.
\begin{prop}\label{pro77.111}
Let $V$ be a   conformal module    over  $\mathfrak{B}(\alpha,\beta,p)$ defined  in Theorem  \ref{77.1}.
\begin{itemize}\parskip-7pt
 \item[{\rm (1)}]
If  $V\cong V_{a,b}$,  then $V$ is irreducible  if and only if $a\neq0$. The module  $V_{0,b}$ contains a unique non-trivial submodule $(\partial+b)V_{0,b}\cong V_{1,b}$.
 \item[{\rm (2)}]
If $V\cong V_{a,b,d}$,  then $V$ is irreducible  if and only if $a\neq0$ or $d\neq0$. The module  $V_{0,b,0}$ contains a unique non-trivial submodule $(\partial+b)V_{0,b,0}\cong V_{1,b,0}$.
  \item[{\rm (3)}]
If $V\cong V_{a,b,c}$,  then $V$ is irreducible  if and only if $a\neq0$ or $c\neq0$. The module  $V_{0,b,0}$ contains a unique non-trivial submodule $(\partial+b)V_{0,b,0}\cong V_{1,b,0}$.
 \item[{\rm (4)}]
If $V\cong V_{a,b,c,d}$,  then $V$ is irreducible  if and only if $a\neq0$  or $c\neq0$  or $d\neq0$.  The module $V_{0,b,0,0}$ contains a unique non-trivial submodule $(\partial+b)V_{0,b,0,0}\cong V_{1,b,0,0}$.
\end{itemize}
\end{prop}

\subsection{Main theorems}
The following result shows that all finite non-trivial irreducible
$\mathfrak{B}(\alpha,\beta,p)$-modules   are free of rank one.
\begin{lemm}\label{lemm789}
Any finite non-trivial irreducible   $\mathfrak{B}(\alpha,\beta,p)$-module $V$ must be free of rank one.
\end{lemm}
\begin{proof}
As we know, any torsion module of $\C[\partial]$ is trivial as a module of Lie conformal algebra. Therefore, any finite non-trivial irreducible   $\mathfrak{B}(\alpha,\beta,p)$-module $V$ must be free as a $\C[\partial]$-module.
By Theorem \ref{5.2},  we see that the $\lambda$-actions  of $L_i$ and $W_i$ on $V$ are trivial for all $i\gg0$.
Let  $k\in\Z_+$ be the largest integer such that the $\lambda$-action  of $\mathfrak{B}(\alpha,\beta,p)_k$  on $V$ is non-trivial.
Then $V$ can be regarded as a    finite non-trivial irreducible  conformal module over $\mathfrak{B}(\alpha,\beta,p)_{[k]}$.
We denote $\mathcal{G}=\{\bar L_{i,m},\bar W_{j,n},\partial\mid i,j\in\Z_+,m,n\in\Z_+\cup\{-1\}\}$. Here we note that   $\{\bar W_{i,-1}\mid i\in\Z_+\}\nsubseteq \mathcal{A}(\mathfrak{B}(\alpha,\beta,p)_{[k]})^e$.  Furthermore, it follows from Proposition \ref{pro2.4} that
  a conformal  $\mathfrak{B}(\alpha,\beta,p)_{[k]}$-module  $V$ can be viewed as a module over the associated extended annihilation algebra
   $\mathcal{G}=\mathcal{A}(\mathfrak{B}(\alpha,\beta,p)_{[k]})^e$   satisfying
\begin{eqnarray}\label{7.1}
\bar L_{i,m}v=\bar W_{j,n}v=\bar W_{j,-1}v=0
\end{eqnarray}
for $0\leq i,j\leq k,m,n\gg0,v\in V$.
Denote
$$\mathcal{G}_z=\{\bar L_{i,m},\bar W_{j,n}\in\mathcal{G}\mid 0\leq i,j\leq k,m,n\geq z-1\},\ z\in\Z_+.$$
Then   $\mathcal{G}_0=\mathcal{A}(\mathfrak{B}(\alpha,\beta,p)_{[k]})$   and $\mathcal{G}\supset\mathcal{G}_0\supset\mathcal{G}_1\cdots$.
From the definition of extended annihilation algebra, it is clear that the element  $\partial\in\mathcal{G}$ satisfies $[\partial, \mathcal{G}_z] =\mathcal{G}_{z-1}$ for $z \geq1$. Denote
$$V_z=\{v\in V\mid \mathcal{G}_zv=0\},\ z\in\Z_+.$$
Clearly, $V_z\neq\emptyset$ for $z\gg0$ by \eqref{7.1}.  Let  $N\in \Z_+$ be the smallest integer such that $V_N\neq\emptyset$.

First we consider  $N=0$. Take $0\neq v\in V_0$. Then $\mathcal{U}(\mathcal{G})v=\C[\partial]\mathcal{U}(\mathcal{G}_0)v=\C[\partial]v$. Thus,
by the irreducibility of $V$, we see that  $V=\C[\partial]v$.
 From $\mathcal{G}_0$ is an ideal of $\mathcal{G}$, we can get  that $\mathcal{G}_0$ acts trivially on $V$.
  By Proposition \ref{pro2.4},  we see that $V$ is a   trivial conformal $\mathfrak{B}(\alpha,\beta,p)$-module, which gives a contradiction.

Next, consider  $N\geq1$.  Choose $0\neq v\in V_N$. We will discuss two cases in the following.
\begin{case}
$(\alpha,\beta)\neq (p,0).$
\end{case}
Then $V_N$ can be seen as a $\mathcal{G}_1/\mathcal{G}_N$-module.
Clearly,  $\mathcal{G}_1/\mathcal{G}_N$ is a finite-dimensional solvable Lie algebra. Because of Lie's
Theorem, there exists a
linear function $\chi$ on $\mathcal{G}_1/\mathcal{G}_N$  such that $x\cdot v = \chi(x) v$ for all $x\in \mathcal{G}_1/\mathcal{G}_N$.
Setting $\mathcal{F}=\mathrm{span}_{\C}\{\bar L_{i,-1}, \partial\mid 0\leq i\leq k\}$,  then  $\mathcal{G}$ has a decomposition of vector
spaces
 \begin{eqnarray*}
 \mathcal{G}=\mathcal{F}\oplus\mathcal{G}_1.
\end{eqnarray*}
By {\it Poincare-Birkhoff-Witt (PBW)} Theorem, the universal enveloping algebra of
$\mathcal{G}$ is
 \begin{eqnarray*}
 U(\mathcal{G})=U(\mathcal{F})\otimes U(\mathcal{G}_1),
\end{eqnarray*}
where $U(\mathcal{F})=\mathrm{span}_{\C}\{\bar L_{0,-1}^{i_0}\bar L_{1,-1}^{i_1}\cdots \bar L_{k,-1}^{i_k} \partial^j\mid i_0,i_1,\ldots,i_k,j\in\Z_+\}$
as a
vector space over $\C$.
Then  we
have
 \begin{eqnarray}\label{vb12}
 V=U(\mathcal{G})\cdot v=U(\mathcal{F}) \cdot v =\sum_{i_0,\ldots,i_k,j\in\Z_+}\C\partial^j\bar L_{0,-1}^{i_0}\bar L_{1,-1}^{i_1}\cdots \bar L_{k,-1}^{i_k} \cdot v.
\end{eqnarray}
Obviously, not all $i\in\Z_+$ satisfy $\bar L_{i,-1}\cdot v=0.$
Otherwise, we can deduce that $V = \C[\partial]v$ is free of
rank one, which contradicts to    $\bar L_{0,-1}v=p(\partial+b)v$ for some $b\in\C$ by Lemma  \ref{5.1} and \eqref{22.111}.

Now we  first   consider $\alpha\neq p$ in this case.
By the definition of  extended annihilation algebra, one can get
\begin{eqnarray}\label{44999}
[\bar L_{i,m},\bar W_{0,0}]=\big((m+1)(\alpha-p)\big)\bar W_{i,m}+\beta(i+p)\bar W_{i,m+1}
\end{eqnarray}
for any $i\in\Z_+,m\in\Z_+\cup\{-1\}$. If $\beta=0$, it is easy to get that $\bar W_{i,m}\cdot v=0$ for $0\leq i\leq k, 0\leq m\leq N-2$.
If $\beta\neq0$, we  set $0\leq i\leq k,m=N-2$ in \eqref{44999}. Then by $x\cdot v = \chi(x) v$ for all $x\in \mathcal{G}_1/\mathcal{G}_N$,
we have $\bar W_{i,N-2}\cdot v=0$ for $0\leq i\leq k$. By recursive method, one has $\bar W_{i,m}\cdot v=0$ for $0\leq i\leq k, 0\leq m\leq N-2$.
Therefore,
 \begin{eqnarray*}
W_i\, {}_\lambda \, v=\sum_{k\in\Z_+}\frac{\lambda^{(k)}}{k!}({ W_i}_{(k)}v)=\sum_{k\in\Z_+}\frac{\lambda^{(k)}}{k!}( W_{i,k} v)=0.
\end{eqnarray*}

Assume that  $R_\partial$ (respectively $L_\partial$) is the right (respectively left) multiplication by $\partial$ in
the universal enveloping algebra of $ \mathcal{G}$. It follows $R_\partial=L_\partial-ad_\partial$ and the binomial
formula that we obtain
  \begin{eqnarray}\label{909}
 \mathcal{G}_N\partial^k&=&R_\partial^k\mathcal{G}_N=(L_\partial-ad_\partial)^k\mathcal{G}_N
\nonumber\\&=&\sum^{k}_{j=0}\partial^{k-j}(-ad_\partial)^j\mathcal{G}_N=\sum^{k}_{j=0}\partial^{k-j}\mathcal{G}_{N-j}
\end{eqnarray}
for $N,k\in\Z_+$.
Since $W_i\, {}_\lambda \, v=0$ and $\C[\partial]\{W_i\mid i\in\Z_+\}$ is an ideal of $\mathfrak{B}(\alpha,\beta,p)$, we check that $W_i\, {}_\lambda \, V=0$   by
\eqref{vb12} and \eqref{909}. Thus, the irreducibility of $V$ as a $\mathfrak{B}(\alpha,\beta,p)$-module is equivalent
to that of $V$ as a $\mathfrak{B}(\alpha,\beta,p)/ \C[\partial]\{W_i\mid i\in\Z_+\}\cong\mathfrak{B}(p)$-module. Then the conclusion holds by Lemma \ref{5.1}.

Consider $\beta\neq0$.
If  $\alpha= p$,
by the definition of  extended annihilation algebra, one gets
\begin{eqnarray*}
[\bar L_{i,m},\bar W_{0,0}]=\beta (i+p) \bar W_{i,m+1}
\end{eqnarray*}
for any $i,m\in\Z_+$.  If $i\neq-p$, then by $x\cdot v = \chi(x) v$  for all $x\in \mathcal{G}_1/\mathcal{G}_N$, we have $\chi(\bar W_{i,m+1})\cdot v=0$ for $i,m\in\Z_+.$
Consider $i=-p$.
For $m\geq0$, we have \begin{eqnarray*}
[\bar L_{0,0},\bar W_{-p,m}]=p(\beta \bar W_{-p,m+1}-(m+1)\bar W_{-p,m}).
\end{eqnarray*}
Taking $m=N-2$ in above relations, we have $\bar W_{-p,N-2}\cdot v=0$. By recursive method, one see that $\bar W_{-p,m}\cdot v=0$ for $ 0\leq m\leq N-2$.
Thus, we can conclude that  $\chi(\bar W_{i,m+1})=0$ for $i,m\in\Z_+.$

 Now  prove $\chi(\bar W_{i,0})=0$ for $i\neq0$. Using
\begin{eqnarray*}
[\bar L_{0,0},\bar W_{j,0}]=j \bar W_{j,0}+\beta p \bar W_{j,1},
\end{eqnarray*}
then by $x\cdot v = \chi(x) v$,  we know that    $\chi(\bar W_{j,0})=0$ for $j\neq0.$
Suppose $\bar W_{0,0}\cdot v=\gamma v$ for $\gamma\in\C$.
By $[\partial,\bar W_{0,0}]=0,[\bar L_{i,-1},\bar W_{0,0}]=\beta(i+p)\bar W_{i,0}$  and \eqref{vb12},
we get
\begin{eqnarray*}
 &&\bar W_{0,0}\cdot\sum_{i_0,\ldots,i_k,j\in\Z_+}\eta_{i_0,\ldots,i_k,j}\partial^j\bar L_{0,-1}^{i_0}\bar L_{1,-1}^{i_1}\cdots \bar L_{k,-1}^{i_k} \cdot v
 \\&=&\sum_{i_0,\ldots,i_k,j\in\Z_+}\eta_{i_0,\ldots,i_k,j}\partial^j(\bar L_{0,-1}-\beta p)^{i_0}\bar W_{0,0}\bar L_{1,-1}^{i_1}\cdots \bar L_{k,-1}^{i_k} \cdot v
 \\&=&\sum_{i_0,\ldots,i_k,j\in\Z_+}\eta_{i_0,\ldots,i_k,j}\partial^j(\bar L_{0,-1}-\beta p)^{i_0}(\bar L_{1,-1}\bar W_{0,0}-\beta(1+p)\bar W_{1,0})\bar L_{1,-1}^{i_1-1}\cdots \bar L_{k,-1}^{i_k} \cdot v
 \\&&\cdots
 \\&=&\gamma\sum_{i_0,\ldots,i_k,j\in\Z_+}\eta_{i_0,\ldots,i_k,j}\partial^j(\bar L_{0,-1}-\beta p)^{i_0}\bar L_{1,-1}^{i_1}\cdots \bar L_{k,-1}^{i_k} \cdot v.
\end{eqnarray*}
By the above computation, we see that the action of $\bar W_{0,0}$ on $V$ can be obtained by  $\mathfrak{B}(p)$-actions.
Then $V$ can be seen as an irreducible $\mathfrak{B}(p)$-module.
Now the conclusion can be
directly obtained by Lemma \ref{5.1}.

\begin{case}
$(\alpha,\beta)=(p,0).$
\end{case}
By the definition of extended annihilation algebra, we obtain that
$\partial-\frac{1}{p}\bar L_{0,-1}$ is the central
element of $\mathcal{G}$. By Schur's Lemma, there exists some $b\in\C$ such that   $\bar L_{0,-1}v=p(\partial+b)v$.
Then it follows from $$\bar L_{i,-1}v=\frac{1}{p}[\bar L_{i,0},\bar L_{0,-1}]v,\ \bar W_{i,-1}v=0$$
that we know that the action of $\mathcal{G}_0$ on  $v$ is determined by $\mathcal{G}_1$ and
$\partial$.  Obviously,    $V_N$ is $\mathcal{G}_1$-invariant. By the irreducibility of $V$ and Lemma \ref{lemm2.5},
 we see that $V=\C[\partial]\otimes_{\C}V_N$  and $V_N$ is a non-trivial irreducible finite-dimensional $\mathcal{G}_1$-module.

If $N=1$,
     we  see that $V_1$ is a trivial $\mathcal{G}_1$-module  by the definition of $V_1$,  which creates a   contradiction.

If $N\geq 2$, it follows from the definition of $V_N$ that  it can be viewed as a $\mathcal{G}_1/ \mathcal{G}_N$-module.
 Note that $\mathcal{G}_1/ \mathcal{G}_N\cong\mathcal{Q}({k,N-2})$. By Lemma \ref{lemm555},
 one can see that   $V_N$ is $1$-dimensional. Then  $V$ is free of rank  one  as a conformal module over $\mathfrak{B}(\alpha,\beta,p)$  by Proposition \ref{pro2.4}.
This proves Lemma \ref{lemm789}.
\end{proof}

Now we present the main result  of this paper, which  shows that the irreducible modules $V$  defined in Theorem  \ref{77.1} exhaust
  all non-trivial finite irreducible conformal modules over $\mathfrak{B}(\alpha,\beta,p)$.
  It can be obtained by Theorem \ref{77.1}, Proposition \ref{pro77.111} and Lemma \ref{lemm789}.
\begin{theo}\label{7.111}
Let $V$ be a non-trivial finite irreducible conformal module over  $\mathfrak{B}(\alpha,\beta,p)$.
\begin{itemize}\parskip-7pt
 \item[{\rm (1)}]
  If $p\neq-1$, and
 \begin{eqnarray*}
\begin{cases}
{\rm (i)}\ (\alpha,\beta)\neq(p,0), \ \mathrm{then}\  V\cong V_{a,b}  &\ \mbox{with}\  a\in\C^*,b\in\C,\\[4pt]
{\rm (ii)}\ (\alpha,\beta)=(p,0),  \ \mathrm{then}\ V\cong V_{a,b,d} &\  \mbox{with} \ a\in\C^* \ \mathrm{or}\  d\in\C^*;
\end{cases}\\
\end{eqnarray*}
 \item[{\rm (2)}]
  If $p=-1$, and
 \begin{eqnarray*}
\begin{cases}
{\rm (iii)}\ (\alpha,\beta)\neq(-1,0), \ \mathrm{then}\  V\cong V_{a,b,c}  &\ \mbox{with}\  a\in\C^* \ \mathrm{or}\  c\in\C^*,\\[4pt]
{\rm (iv)}\ (\alpha,\beta)=(-1,0),  \ \mathrm{then}\ V\cong V_{a,b,c,d} &\  \mbox{with} \ a\in\C^* \ \mathrm{or}\  c\in\C^* \ \mathrm{or}\  d\in\C^*.
\end{cases}\\
\end{eqnarray*}
\end{itemize}
\end{theo}

\section{Applications}
By  the definition  of  \eqref{bn3.2},  one can see that  $\mathfrak{b}(n)$ for $n>0$ has a  $\C[\partial]$-basis $\{\bar L_{i},\bar W_{i}\mid 0\leq i\leq n\}$ with  the following
non-vanishing  $\lambda$-brackets:
\begin{eqnarray*}
&& [\bar L_i\, {}_\lambda \,\bar L_j]=\big((i-n)\partial+(i+j-2n)\lambda \big)\bar L_{i+j},
\\&&
[\bar L_i\, {}_\lambda \, \bar W_j]=\Big((i-n)(\partial+\beta)+(i+j+\alpha)\lambda\Big) \bar W_{i+j}
\end{eqnarray*}
for any $\alpha,\beta\in\C$ ($i+j>n$ the above relations are trivial).
The following   $\C[\partial]$-modules $\bar V_{a,b},\bar V_{a,b,d}$ are  conformal modules over  $\mathfrak{b}(n)$.
\begin{itemize}\parskip-7pt
 \item[{\rm (1)}] $\bar V_{a,b}=\C[\partial]v$ with  relations
 $\bar W_i\,{}_\lambda\, v=0,0\leq i\leq n$ and

 \begin{eqnarray*}
\bar L_i\,{}_\lambda\, v=
\begin{cases}
 \ -n(\partial+a\lambda+b)v,  &\ \mbox{if}\    i=0;\\[4pt]
 \ 0, &\  \mbox{if} \   1\leq i\leq n
\end{cases}
\end{eqnarray*}
 for $a,b\in\C$, if  $(\alpha,\beta)\neq(-n,0)$;
 \item[{\rm (2)}] $\bar V_{a,b,d}=\C[\partial]v$ with    relations
  \begin{eqnarray*}
\bar L_i\,{}_\lambda\, v=
\begin{cases}
 \ -n(\partial+a\lambda+b)v,  &\ \mbox{if}\    i=0;\\[4pt]
 \ 0, &\  \mbox{if} \   1\leq i\leq n;\\
\end{cases}\quad
\bar W_i\,{}_\lambda\, v=
\begin{cases}
 \ dv,  &\ \mbox{if}\    i=0;\\[4pt]
 \ 0, &\  \mbox{if} \   1\leq i\leq n
\end{cases}
\end{eqnarray*}
 for  $a,b,d\in\C$, if $(\alpha,\beta)=(-n,0)$.
\end{itemize}
The following   $\C[\partial]$-modules   $\bar V_{a,b,c},\bar V_{a,b,c,d}$ are conformal modules over  $\mathfrak{b}(1)$.
\begin{itemize}\parskip-7pt
 \item[{\rm (3)}] $\bar V_{a,b,c}=\C[\partial]v$ with   relations
  $\bar W_i\,{}_\lambda\, v=0,i=0,1$ and

 \begin{eqnarray*}
\bar L_i\,{}_\lambda\, v=
\begin{cases}
 \ - (\partial+a\lambda+b)v,  &\ \mbox{if}\    i=0;\\[4pt]
 \ cv, &\  \mbox{if} \    i=1
\end{cases}
\end{eqnarray*}
 for $a,b,c\in\C$, if  $(\alpha,\beta)\neq(-1,0)$;
 \item[{\rm (4)}] $\bar V_{a,b,c,d}=\C[\partial]v$ with   relations
 \begin{eqnarray*}
\bar L_i\,{}_\lambda\, v=
\begin{cases}
 \ -(\partial+a\lambda+b)v,  &\ \mbox{if}\    i=0;\\[4pt]
 \ cv, &\  \mbox{if} \     i=1;\\
\end{cases}\quad
\bar W_i\,{}_\lambda\, v=
\begin{cases}
 \ dv,  &\ \mbox{if}\    i=0;\\[4pt]
 \ 0, &\  \mbox{if} \   i=1
\end{cases}
\end{eqnarray*}
 for  $a,b,c,d\in\C$, if $(\alpha,\beta)=(-1,0)$.
\end{itemize}
By Theorem \ref{77.1}, we have the following corollary.
\begin{coro}\label{coro666}
Let $\bar V$ be a non-trivial free conformal module of rank one  over  $\mathfrak{b}(n)$.
\begin{itemize}\parskip-7pt
 \item[{\rm (1)}]
  If $n>1$, and
 \begin{eqnarray*}
\begin{cases}
{\rm (i)}\ (\alpha,\beta)\neq(-n,0), \ \mathrm{then}\ \bar V\cong \bar V_{a,b}  &\ \mbox{with}\  a,b\in\C,\\[4pt]
{\rm (ii)}\ (\alpha,\beta)=(-n,0),  \ \mathrm{then}\ \bar V\cong \bar V_{a,b,d} &\  \mbox{with} \ a,b,d\in\C;
\end{cases}\\
\end{eqnarray*}
 \item[{\rm (2)}]
  If $n=1$, and
 \begin{eqnarray*}
\begin{cases}
{\rm (iii)}\ (\alpha,\beta)\neq(-1,0), \ \mathrm{then}\ \bar V\cong \bar V_{a,b,c}  &\ \mbox{with}\  a,b,c\in\C,\\[4pt]
{\rm (iv)}\ (\alpha,\beta)=(-1,0),  \ \mathrm{then}\ \bar  V\cong \bar V_{a,b,c,d} &\  \mbox{with} \ a,b,c,d\in\C.
\end{cases}\\
\end{eqnarray*}
\end{itemize}
\end{coro}
 Furthermore, for the above modules we have the same irreducibility assertions as those for $\mathfrak{B}(\alpha,\beta,-n)$-modules in Proposition \ref{pro77.111}.
  The irreducible modules in Corollary \ref{coro666} exhaust all non-trivial finite irreducible conformal modules over $\mathfrak{b}(n)$.
\begin{coro}
Let $\bar V$ be a finite non-trivial irreducible conformal module   over  $\mathfrak{b}(n)$.
\begin{itemize}\parskip-7pt
 \item[{\rm (1)}]
  If $n>1$, and
 \begin{eqnarray*}
\begin{cases}
{\rm (i)}\ (\alpha,\beta)\neq(-n,0), \ \mathrm{then}\ \bar V\cong \bar V_{a,b}  &\ \mbox{with}\  a\in\C^*,b\in\C,\\[4pt]
{\rm (ii)}\ (\alpha,\beta)=(-n,0),  \ \mathrm{then}\ \bar V\cong \bar V_{a,b,d} &\  \mbox{with} \ a\in\C^* \  \mathrm{or} \ d\in\C^*;
\end{cases}\\
\end{eqnarray*}
 \item[{\rm (2)}]
  If $n=1$, and
 \begin{eqnarray*}
\begin{cases}
{\rm (iii)}\ (\alpha,\beta)\neq(-1,0), \ \mathrm{then}\ \bar V\cong \bar V_{a,b,c}  &\ \mbox{with}\  a\in\C^* \  \mathrm{or} \ c\in\C^*,\\[4pt]
{\rm (iv)}\ (\alpha,\beta)=(-1,0),  \ \mathrm{then}\ \bar  V\cong \bar V_{a,b,c,d} &\  \mbox{with} \ a\in\C^* \  \mathrm{or} \ c\in\C^* \  \mathrm{or} \ d\in\C^*.
\end{cases}\\
\end{eqnarray*}
\end{itemize}
\end{coro}

\section*{Acknowledgements}
This work was supported by the National Natural Science Foundation of China (No.11801369,   11871421, 11431010)
and the Scientific Research Foundation of Hangzhou Normal University (No. 2019QDL012).
The authors would like  to thank Profs. X. Dai and C. Xia    for a lot of helpful discussions when
they were preparing the paper.

\bigskip

Haibo Chen
\vspace{2pt}

  School of  Statistics and Mathematics, Shanghai Lixin University of  Accounting and Finance,   Shanghai
201209, China

\vspace{2pt}
rebel1025@126.com

\bigskip

Yanyong Hong
\vspace{2pt}

Department of Mathematics, Hangzhou Normal University
Hangzhou 311121,  China

\vspace{2pt}
hongyanyong2008@yahoo.com

\bigskip

Yucai Su
\vspace{2pt}

  School of Mathematical Sciences, Tongji University, Shanghai
200092, China
\vspace{2pt}

ycsu@tongji.edu.cn
\end{document}